\documentclass[12pt]{amsart}%
\usepackage{amscd,amsmath,latexsym,amsthm,amsfonts,amssymb,graphicx,color,geometry}
\usepackage{enumerate}
\usepackage[pdftex,plainpages=false,pdfpagelabels,
colorlinks=true,citecolor=black,linkcolor=black,urlcolor=black,
filecolor=black,bookmarksopen=true,unicode=true]{hyperref}
\usepackage{amsmath}
\usepackage{amsfonts}
\usepackage{amssymb}
\usepackage[T1]{fontenc}
\usepackage{graphicx}
\usepackage{xcolor}
\usepackage{framed}
\usepackage{colortbl}%
\providecommand{\U}[1]{\protect\rule{.1in}{.1in}}
\newtheorem{theorem}{Theorem}[section]

\numberwithin{equation}{section}
\begin{document}
\title{A note on the geometry of certain classes of linear operators}

\begin{abstract}
In this note we introduce a new technique to answer an issue posed in
\cite{FPT} concerning geometric properties of the set of non-surjective linear
operators. We also extend and improve a related result from the same paper.

\end{abstract}
\author[D. Diniz]{Diogo Diniz}
\address{Unidade Acad\^{e}mica de Matem\'{a}tica e Estat\'{\i}stica \\
Universidade Federal de Campina Grande \\
58109-970 - Campina Grande, Brazil.}
\email{diogo@mat.ufcg.edu.br and diogodpss@gmail.com}
\author[A. Raposo Jr.]{Anselmo Raposo Jr.}
\address{Departamento de Matem\'{a}tica \\
Universidade Federal do Maranh\~{a}o \\
65085-580 - S\~{a}o Lu\'{\i}s, Brazil.}
\email{anselmo.junior@ufma.br}
\thanks{D. Diniz was partially supported by CNPq 301704/2019-8 and Grant 2019/0014
Paraiba State Research Foundation (FAPESQ)}
\maketitle

\section{Introduction}

In 1872 K. Weierstrass constructed an example af a nowhere differentiable
continuous function from $[0,1]$ on $\mathbb{R}$. This non intuitive result,
now known as Weierstrass Monster, was pushed further in 1966, when V. Gurariy
constructed an infinite-dimensional subspace formed, except for the null
vector, by continuous nowhere differentiable functions. In 2004, Aron, Gurariy
and Seoane \cite{AGSS} investigated similar problems in other settings,
initiating the field of research known as \textquotedblleft
lineability\textquotedblright: the idea is to look for linear structures
inside exotic subsets of vector spaces. If $V$ is a vector space and $\alpha$
is a cardinal number, a subset $A$ of $V$ is called $\alpha\text{-lineable}$
in $V$ if $A\cup\left\{  0\right\}  $ contains an $\alpha$-dimensional linear
subspace $W$ of $V$. When $V$ has a topology and the subspace $W$ can be
chosen to be closed, we say that $A$ is spaceable. We refer to the book
\cite{book} for a general panorama of the subject.

As a matter of fact, with the development of the theory, it was observed that
positive results of lineability were quite common, although general techniques
are, in general, not available. Towards a more demanding notion of linearity,
F\'{a}varo, Pellegrino and Tomaz introduced a more involved geometric concept:
let $\alpha$, $\beta$ and $\lambda$ be cardinal numbers, with $\alpha
<\beta\leq\lambda$, and let $V$ be a vector space such that $\dim V=\lambda$.
A subset $A$ of $V$ is called $\left(  \alpha,\beta\right)  $-lineable if, for
every subspace $W_{\alpha}\subset V$ such that $\dim W_{\alpha}=\alpha$ and
$W_{\alpha}\subset A\cup\left\{  0\right\}  $\ there is a subspace $W_{\beta
}\subset V$ with $\dim W_{\beta}=\beta$ and $W_{\alpha}\subset W_{\beta
}\subset A\cup\left\{  0\right\}  $. When $V$ is a topological vector space,
we shall say that $A$ is $\left(  \alpha,\beta\right)  $-spaceable when the
subspace $W_{\beta}$ can be chosen to be closed .

A well-known technique in lineability is known as \textquotedblleft mother
vector technique\textquotedblright: it consists of choosing a vector $v$ in
the set $A$ and generating a subspace $W\subset A\cup\{0\}$ containing
\textquotedblleft copies\textquotedblright\ of $v$. However, in general, the
vector $v$ does not belong to the generated subspace (see, for instance,
\cite{PT}). Constructing a vector space of prescribed dimension and containing
an arbitrary given vector is a rather more involving problem, which is
probably another motivation of this more strict approach to lineability.

Under this new perspective, several simple problems, from the point of view of
ordinary lineability, gain more subtle contours. For instance, it is obvious
that the set of the continuous linear operators $u\colon\ell_{p}%
\rightarrow\ell_{q}$ that are non-surjective is $\mathfrak{c}$-spaceable (here
and henceforth $\mathfrak{c}$ denotes the continuum). In fact, if $\pi
_{1}\colon\ell_{q}\rightarrow\mathbb{K}$ is the projection at the first
coordinate, just consider the colection of the continuous and non-surjective
linear operators for which $\pi_{1}\circ u\equiv0$. Therefore, only $\left(
\alpha,\beta\right)  $-lineability matters in this framework.

In this note we answer a question posed in \cite{FPT} on the $\left(
1,\mathfrak{c}\right)  $-lineability of a certain set of non surjective
functions. Our solution uses a technique that, to the best of the authors
knowledge, is new.

\section{Lineability vs injectivity and surjectivity}

Lineability properties of the sets of injective and surjective continuous
linear operators between classical sequence spaces were recently investigated
by \cite{ABGJRMFSS} and \cite{DFPR}. In \cite[Theorem 3.1]{FPT} the authors
investigated more subtle geometric properties in the setting of non injective
continuous linear operators by proving that if $p,q\geq1$ and
\begin{equation}
A:=\left\{  u\colon\ell_{p}\rightarrow\ell_{q}:u\text{ is linear, continuous
and non injective}\right\}  \text{,}\label{c1c1c1}%
\end{equation}
then $A$ is $\left(  1,\mathfrak{c}\right)  $-lineable. In the same paper the
authors pose a question on the $\left(  1,\mathfrak{c}\right)  $-lineability
of the set
\begin{equation}
D:=\left\{  u\colon\ell_{p}\rightarrow\ell_{q}:u\text{\ is linear, continuous
and non-surjective}\right\}  \text{.}\label{a1a1a1}%
\end{equation}
In this section we shall show that $D$ is $\left(  1,\mathfrak{c}\right)
$-lineable and, as a matter of fact, our technique works in a more general
environment of sequence spaces. We shall say that a Banach sequence space $E$
of $X$-valued sequences where $X$ is a Banach space is reasonable if
$c_{00}\left(  X\right)  \subset E$ and, for all $x=\left(  x_{j}\right)
_{j=1}^{\infty}\in E$ and $\left(  \alpha_{j}\right)  _{j=1}^{\infty}\in
\ell_{\infty}$, we have $\left(  \alpha_{j}x_{j}\right)  _{j=1}^{\infty}\in E$
with
\begin{equation}
\left\Vert \left(  \alpha_{j}x_{j}\right)  _{j=1}^{\infty}\right\Vert
\leq\left\Vert \left(  \alpha_{j}\right)  _{j=1}^{\infty}\right\Vert _{\infty
}\left\Vert \left(  x_{j}\right)  _{j=1}^{\infty}\right\Vert \text{.}%
\label{b1b1b1}%
\end{equation}

The class of reasonable sequence spaces includes various classical sequence
spaces. For instance, for $1<p<\infty$, the $\ell_{p}\left(  X\right)  $
spaces of $p$-summable sequences, the $\ell_{p}^{w}\left(  X\right)  $ spaces
of weakly $p$-summable sequences and the $\ell_{p}^{u}\left(  X\right)  $
spaces of unconditionally $p$-summable sequences are reasonable sequence
spaces. The spaces $\ell_{\infty}\left(  X\right)  $, $c_{0}\left(  X\right)
$, of bounded and null sequences, respectively and the Lorentz spaces
$\ell_{\left(  w,p\right)  }\left(  X\right)  $ are also reasonable sequence spaces.

Our result reads as follows:

\begin{theorem}
\label{Teo3.1}Let $V\neq\left\{  0\right\}  $ be a normed vector space and
$X\neq\left\{  0\right\}  $ be a Banach space. Let $E$ be a reasonable
sequence space of $X$-valued sequences. The set%
\[
D_{V,E}=\left\{  u\colon V\rightarrow E:u\text{\ is linear, continuous and
non-surjective}\right\}
\]
is $\left(  1,\mathfrak{c}\right)  $-lineable.
\end{theorem}

\section{The proof}

Fixed $v\in D_{V,E}\backslash\left\{  0\right\}  $, let
\[
\mathbb{N}_{v}=\left\{  k\in\mathbb{N}:\pi_{k}\circ v\not \equiv 0\right\}
\text{,}%
\]
where%
\begin{align*}
\pi_{k}\colon E &  \rightarrow X\\
\left(  x_{j}\right)  _{j=1}^{\infty} &  \mapsto x_{k}%
\end{align*}
is the $k$-th projection over $X$. If $\mathbb{N}_{v}$ is a proper subset of
$\mathbb{N}$, the proof is simple. In fact, if $j_{0}\in\mathbb{N}%
\backslash\mathbb{N}_{v}$, since $c_{00}(X)\subset E$, it is obvious that the
subspace%
\[
N:=\left\{  u\colon V\rightarrow E:u\text{ is linear, continuous and }%
\pi_{j_{0}}\circ u\equiv0\right\}
\]
is contained in $D_{V,E}$ and it is also plain that $v\in N$. We will prove
that $\dim N\geq\mathfrak{c}$. Since $v$ is not identically zero, there exists
$x_{0}\in V$ such that $v\left(  x_{0}\right)  =w_{0}\neq0$. By the
Hahn-Banach Theorem, there is a continuous linear functional $\varphi\colon
E\rightarrow\mathbb{K}$, such that $\varphi\left(  w_{0}\right)  =1$. Fixing
$a\in X\backslash\left\{  0\right\}  $, for each $k\in\mathbb{N}$, let us
define%
\[
w_{k}=(\underset{j_{0}+k-1}{\underbrace{0,\ldots,0}},a,0,0,\ldots)\in E
\]
and consider the linear operators $T_{k}\colon E\rightarrow E$ given by%
\[
T_{k}\left(  w\right)  =\varphi\left(  w\right)  \dfrac{w_{k}}{\left\Vert
w_{k}\right\Vert }\text{.}%
\]
Obviously, the operators $T_{k}$ are continuous and%
\[
\left\Vert T_{k}\right\Vert =\sup_{\left\Vert w\right\Vert \leq1}\left\Vert
\varphi\left(  w\right)  \dfrac{w_{k}}{\left\Vert w_{k}\right\Vert
}\right\Vert =\sup_{\left\Vert w\right\Vert \leq1}\left\vert \varphi\left(
w\right)  \right\vert =\left\Vert \varphi\right\Vert \text{.}%
\]
Hence, the operators $S_{k}=T_{k}\circ v$ are continuous too and%
\[
\left\Vert S_{k}\right\Vert =\left\Vert T_{k}\circ v\right\Vert \leq\left\Vert
T_{k}\right\Vert \left\Vert v\right\Vert =\left\Vert \varphi\right\Vert
\left\Vert v\right\Vert \text{.}%
\]
Notice that $S_{k}\left(  x_{0}\right)  =w_{k}/\left\Vert w_{k}\right\Vert $
and, consequently,%
\[
\pi_{j_{0}+k}\circ S_{k}\not \equiv 0\text{.}%
\]
Hence, $S_{k}\in N$ for each $k\in\mathbb{N}$. It is obvious that the set
\[
\left\{  S_{k}:k\in\mathbb{N}\right\}  \subset N
\]
is linearly independent. Define%
\begin{align*}
\Psi\colon\ell_{1} &  \rightarrow\mathcal{L}\left(  V;E\right)  \\
\left(  a_{n}\right)  _{n=1}^{\infty} &  \mapsto%
{\textstyle\sum_{n=1}^{\infty}}
a_{n}S_{n}\text{.}%
\end{align*}
Since $\Psi$ is well-defined, linear and injective we have%
\[
\dim\Psi(\ell_{1})=\mathfrak{c}%
\]
and since $\Psi(\ell_{1})\subset N$ the proof of the case $\mathbb{N}_{v}%
\neq\mathbb{N}$ is done.

Now, let us suppose that $\mathbb{N}_{v}=\mathbb{N}$. By (\ref{b1b1b1}) we
know that, for each $\left(  \alpha_{n}\right)  _{n=1}^{\infty}\in\ell
_{\infty}$,%
\begin{align*}
S_{\left(  \alpha_{n}\right)  _{n=1}^{\infty}}^{v}\colon V  &  \rightarrow E\\
x  &  \mapsto\left(  \alpha_{n}\left(  v\left(  x\right)  \right)
_{n}\right)  _{n=1}^{\infty}%
\end{align*}
is a well-defined continuous linear operator. It is plain that $S_{\left(
\alpha_{n}\right)  _{n=1}^{\infty}}^{v}\in D_{V,E}$ whenever $\left(
\alpha_{n}\right)  _{n=1}^{\infty}$ is a sequence in $\ell_{\infty}$ having
some null entry (because, if $\alpha_{i}=0$, then the $i$-th coordinate of
$S_{\left(  \alpha_{n}\right)  _{n=1}^{\infty}}^{v}\left(  x\right)  $ is zero
for all $x\in V$). Let us consider, therefore, $\left(  \alpha_{n}\right)
_{n=1}^{\infty}\in\ell_{\infty}$ such that $\alpha_{n}\neq0$ for all
$n\in\mathbb{N}$ and fix $w=\left(  w_{n}\right)  _{n=1}^{\infty}\in
E\backslash v\left(  V\right)  $. Since $\left(  \alpha_{n}w_{n}\right)
_{n=1}^{\infty}\in E$, we have
\[
\left(  \alpha_{n}w_{n}\right)  _{n=1}^{\infty}\in E\backslash S_{\left(
\alpha_{n}\right)  _{n=1}^{\infty}}^{v}\left(  V\right)  \text{.}%
\]
In fact, if there was $x\in V$ such that $S_{\left(  \alpha_{n}\right)
_{n=1}^{\infty}}^{v}\left(  x\right)  =\left(  \alpha_{n}w_{n}\right)
_{n=1}^{\infty}$, we would have%
\[
\left(  \alpha_{n}w_{n}\right)  _{n=1}^{\infty}=\left(  \alpha_{n}\left(
v\left(  x\right)  \right)  _{n}\right)  _{n=1}^{\infty}%
\]
and, since $\alpha_{n}\neq0$ for all $n\in\mathbb{N}$, we would have $v\left(
x\right)  =w$, which is impossible. Hence, $S_{\left(  \alpha_{n}\right)
_{n=1}^{\infty}}^{v}\in D_{V,E}$.

Now consider the linear map%
\begin{align*}
\Lambda\colon\ell_{\infty}  &  \rightarrow\mathcal{L}\left(  V;E\right) \\
\Lambda\left(  \left(  \mu_{n}\right)  _{n=1}^{\infty}\right)   &  =S_{\left(
\mu_{n}\right)  _{n=1}^{\infty}}^{v}\text{.}%
\end{align*}
We have just proved that $S_{\left(  \alpha_{n}\right)  _{n=1}^{\infty}}%
^{v}\in D_{V,E}$ for all $\left(  \alpha_{n}\right)  _{n=1}^{\infty}\in
\ell_{\infty}$; thus
\[
\Lambda\left(  \ell_{\infty}\right)  \subset D_{V,E}.
\]
Note that $\Lambda$ is injective. In fact, if $\left(  \mu_{n}\right)
_{n=1}^{\infty}\in\ell_{\infty}$ and $\Lambda\left(  \left(  \mu_{n}\right)
_{n=1}^{\infty}\right)  =0$, since $\mathbb{N}_{v}=\mathbb{N}$, it follows
that, for all $k\in\mathbb{N}$, there is $x^{\left(  k\right)  }\in V$ such
that $\left(  v\left(  x^{\left(  k\right)  }\right)  \right)  _{k}\neq0$.
However, the $k$-th coordinate of $S_{\left(  \mu_{n}\right)  _{n=1}^{\infty}%
}^{v}\left(  x^{\left(  k\right)  }\right)  $ is $\mu_{k}\left(  v\left(
x^{\left(  k\right)  }\right)  \right)  _{k}$, which must be null and,
consequently, $\mu_{k}=0$ for all $k$ and $\left(  \mu_{n}\right)
_{n=1}^{\infty}=0$.

Observe that, choosing $\left(  \lambda_{n}\right)  _{n=1}^{\infty}=\left(
1,1,1,\ldots\right)  \in\ell_{\infty}$, then $v=\Lambda\left(  \left(
\lambda_{n}\right)  _{n=1}^{\infty}\right)  \in\Lambda\left(  \ell_{\infty
}\right)  $. Since $\Lambda$ is injective, we have%
\[
\dim\left(  \Lambda\left(  \ell_{\infty}\right)  \right)  =\dim\left(
\ell_{\infty}\right)  =\mathfrak{c}\text{,}%
\]
and the proof is done.

\section{Final remarks}

We finish this note by showing that the previous technique gives us the
following improvement of \cite[Theorem 3.1]{FPT}:

\begin{theorem}
Let $V\neq\left\{  0\right\}  $ be a normed vector space and $X\neq\left\{
0\right\}  $ be a Banach space. Let $E$ be a Banach sequence space such that
$c_{00}\left(  X\right)  \subset E$. If%
\[
A_{V,E}:=\left\{  u\colon V\rightarrow E:u\text{ is linear, continuous and
non-injective}\right\}  \neq\left\{  0\right\}
\]
then $A_{V,E}$ is $\left(  1,\beta\right)  $-spaceable, where $\beta
=\max\left\{  \mathfrak{c},\dim X\right\}  $.
\end{theorem}

\begin{proof}
Fixed $v\in A_{V,E}\backslash\left\{  0\right\}  $, let $x_{0},y_{0}\in V$,
with $x_{0}\neq y_{0}$, be such that%
\[
v(x_{0})=\left(  \left(  v\left(  x_{0}\right)  \right)  _{n}\right)
_{n=1}^{\infty}=\left(  \left(  v\left(  y_{0}\right)  \right)  _{n}\right)
_{n=1}^{\infty}=v(y_{0})\text{.}%
\]
It is obvious that the subspace%
\[
M:=\left\{  u\colon V\rightarrow E:u\text{ is linear, continuous and }u\left(
x_{0}\right)  =u\left(  y_{0}\right)  \right\}
\]
is contained in $A_{V,E}$ and $v\in M$. It is sufficient to show that $M$ is
closed and $\dim M\geq\beta$. Since $v$ is not identically zero, there exists
$\xi_{0}\in V$ such that $v\left(  \xi_{0}\right)  =w_{0}\neq0$. By the
Hahn-Banach Theorem, there is a continuous linear functional $\varphi\colon
E\rightarrow\mathbb{K}$, such that $\varphi\left(  w_{0}\right)  =1$. Let
$\left\{  a_{\gamma}:\gamma\in\Gamma\right\}  $ be a Hamel basis of $V$. For
each $\gamma\in\Gamma$ and each $k\in\mathbb{N}$, let us define%
\[
w_{k}^{\gamma}=(\underset{k-1}{\underbrace{0,\ldots,0}},a_{\gamma}%
,0,0,\ldots)\in E
\]
and consider the linear operators $T_{k}^{\gamma}\colon E\rightarrow E$ given
by%
\[
T_{k}^{\gamma}\left(  w\right)  =\varphi\left(  w\right)  \dfrac{w_{k}%
^{\gamma}}{\left\Vert w_{k}^{\gamma}\right\Vert }\text{.}%
\]
Obviously, the operators $T_{k}^{\gamma}$ are continuous and, thus, the
operators $R_{k}^{\gamma}=T_{k}^{\gamma}\circ v$ are continuous too. Notice
that $R_{k}^{\gamma}\left(  \xi_{0}\right)  =w_{k}^{\gamma}/\left\Vert
w_{k}^{\gamma}\right\Vert $ and, consequently,%
\[
\pi_{k}\circ R_{k}^{\gamma}\not \equiv 0
\]
and thus, $R_{k}^{\gamma}\in M$ for each $k\in\mathbb{N}$. Let us see that
\[
\left\{  R_{k}^{\gamma}:k\in\mathbb{N},\gamma\in\Gamma\right\}  \subset M
\]
is linearly independent. In fact, let $\left(  k_{1},\gamma_{1}\right)
,\ldots,\left(  k_{n},\gamma_{n}\right)  \in\mathbb{N}\times\Gamma$ be
pairwise distinct and let $\lambda_{1},\ldots,\lambda_{n}\in\mathbb{K}$ such
that%
\[
\lambda_{1}R_{k_{1}}^{\gamma_{1}}+\cdots+\lambda_{n}R_{k_{n}}^{\gamma_{n}%
}=0\text{.}%
\]
Then%
\[
0=\lambda_{1}R_{k_{1}}^{\gamma_{1}}\left(  \xi_{0}\right)  +\cdots+\lambda
_{n}R_{k_{n}}^{\gamma_{n}}\left(  \xi_{0}\right)  =\lambda_{1}\frac{w_{k_{1}%
}^{\gamma_{1}}}{\left\Vert w_{k_{1}}^{\gamma_{1}}\right\Vert }+\cdots
+\lambda_{n}\frac{w_{k_{n}}^{\gamma_{n}}}{\left\Vert w_{k_{n}}^{\gamma_{n}%
}\right\Vert }%
\]
and it is plain that $\lambda_{1}=\cdots=\lambda_{n}=0$ if $k_{i}\neq k_{j}$
whenever $i\neq j$, with $i,j\in\left\{  1,\ldots,n\right\}  $. With no lost
of generality, assuming that $k=k_{1}=\cdots=k_{p_{1}}$, $p_{1}\leq n$, and
$k_{i}\neq k$ if $i>p_{1}$, we have that the $k$-th coordinate of
\[
\lambda_{1}\frac{w_{k_{1}}^{\gamma_{1}}}{\left\Vert w_{k_{1}}^{\gamma_{1}%
}\right\Vert }+\cdots+\lambda_{n}\frac{w_{k_{n}}^{\gamma_{n}}}{\left\Vert
w_{k_{n}}^{\gamma_{n}}\right\Vert }%
\]
is%
\[
0=\frac{\lambda_{1}}{\left\Vert w_{k_{1}}^{\gamma_{1}}\right\Vert }%
a_{\gamma_{1}}+\cdots+\frac{\lambda_{p_{1}}}{\left\Vert w_{k_{p_{1}}}%
^{\gamma_{p_{1}}}\right\Vert }a_{\gamma_{p_{1}}}\text{.}%
\]
By hypothesis, $k_{1}=\cdots=k_{p_{1}}$ implies $\gamma_{1},\ldots
,\gamma_{p_{1}}$ pairwise distinct and, therefore, $a_{\gamma_{1}},\ldots,$
$a_{\gamma_{p_{1}}}$ are linearly independent. Hence, $\lambda_{1}%
=\cdots=\lambda_{p_{1}}=0$ and%
\[
0=\lambda_{p_{1}+1}\frac{w_{k_{p_{1}+1}}^{\gamma_{p_{1}+1}}}{\left\Vert
w_{k_{p_{1}+1}}^{\gamma_{p_{1}+1}}\right\Vert }+\cdots+\lambda_{n}%
\frac{w_{k_{n}}^{\gamma_{n}}}{\left\Vert w_{k_{n}}^{\gamma_{n}}\right\Vert
}\text{.}%
\]
Again, with no lost of generality, assuming that $k=k_{p_{1}+1}=\cdots
=k_{p_{2}}$, $p_{2}\leq n$, and $k_{i}\neq k$ if $i>p_{2}$, we have that the
$k$-th coordinate of
\[
\lambda_{p_{1}+1}\frac{w_{k_{p_{1}+1}}^{\gamma_{p_{1}+1}}}{\left\Vert
w_{k_{p_{1}+1}}^{\gamma_{p_{1}+1}}\right\Vert }+\cdots+\lambda_{n}%
\frac{w_{k_{n}}^{\gamma_{n}}}{\left\Vert w_{k_{n}}^{\gamma_{n}}\right\Vert }%
\]
is%
\[
0=\frac{\lambda_{p_{1}+1}}{\left\Vert w_{k_{p_{1}+1}}^{\gamma p_{1}%
+1}\right\Vert }a_{\gamma_{p_{1}+1}}+\cdots+\frac{\lambda_{p_{2}}}{\left\Vert
w_{k_{p_{2}}}^{\gamma_{p_{2}}}\right\Vert }a_{\gamma_{p_{2}}}\text{.}%
\]
By hypothesis, $k_{p_{1}+1}=\cdots=k_{p_{2}}$ implies $\gamma_{p_{1}+1}%
,\ldots,\gamma_{p_{2}}$ pairwise distinct and, therefore, $a_{\gamma_{p_{1}%
+1}},\ldots,$ $a_{\gamma_{p_{2}}}$ are linearly independent. Hence,
$\lambda_{p_{1}+1}=\cdots=\lambda_{p_{2}}=0$. Proceeding in this way, after
finitely many steps, or we get $\lambda_{1}=\cdots=\lambda_{n}=0$ or we obtain
$m<n$ such that $\lambda_{1}=\cdots=\lambda_{m}=0$ and $k_{m+1},\ldots,k_{n}$
are pairwise distinct. So we have%
\[
0=\lambda_{m+1}\frac{w_{k_{m+1}}^{\gamma_{m+1}}}{\left\Vert w_{k_{m+1}%
}^{\gamma_{m+1}}\right\Vert }+\cdots+\lambda_{n}\frac{w_{k_{n}}^{\gamma_{n}}%
}{\left\Vert w_{k_{n}}^{\gamma_{n}}\right\Vert }%
\]
and, as we know, this implies $\lambda_{m+1}=\cdots=\lambda_{n}=0$. Notice
that, at this moment, we have shown that%
\begin{equation}
\dim M\geq\max\left\{  \aleph_{0},\dim X\right\}  \text{.}\label{d1d1d1}%
\end{equation}
Supposing that $u$ lies in the closure of $M$, let $\left(  u_{n}\right)
_{n=1}^{\infty}$ be a sequence in $M$ such that $\lim\limits_{n\rightarrow
\infty}u_{n}=u$ (we are considering, as usual, the canonical $\sup$ norm in
the set of continuous linear operators $\mathcal{L}\left(  V;E\right)  $ from
$V$ to $E$). Since%
\[
u\left(  x_{0}\right)  =\lim_{n\rightarrow\infty}u_{n}\left(  x_{0}\right)
=\lim_{n\rightarrow\infty}u_{n}\left(  y_{0}\right)  =u\left(  y_{0}\right)
\text{,}%
\]
we conclude that $u\in M$. Thus, $M$ is closed in $\mathcal{L}\left(
V;E\right)  $ and $\dim M\geq\mathfrak{c}$. Since $\mathfrak{c}>\aleph_{0}$,
from here and (\ref{d1d1d1}) we get $\dim M\geq\beta$.
\end{proof}

We recall that, as commented in \cite{FPT} it is not true that $\left(
1,\mathfrak{c}\right)  $-lineability is inherited by inclusions. So, the
following result, which is proved by combinations of the previous techniques,
shall be noticed.

\begin{theorem}
Let $V$, $X$ and $E$ be as in Theorem \ref{Teo3.1}. The set%
\[
C_{V,E}:=\left\{  u\colon V\rightarrow E:u\text{ is linear, continuous,
non-surjective and non-injective}\right\}
\]
is $\left(  1,\mathfrak{c}\right)  $-lineable.
\end{theorem}

\end{document}